\documentclass[11pt,a4paper,leqno]{amsart}

\usepackage[latin1]{inputenc}
\usepackage[T1]{fontenc}
\usepackage{amsfonts}
\usepackage{amsmath}
\usepackage{amssymb}
\usepackage{eurosym}
\usepackage{mathrsfs}
\usepackage{palatino}
\usepackage{color}
\usepackage{esint}
\usepackage{url}
\usepackage{verbatim}

\usepackage{enumerate}

\newcommand{\dtr}{\mathcal{D}^{tr}}

\newcommand{\R}{\mathbb{R}}
\newcommand{\C}{\mathbb{C}}

\newcommand{\Z}{\mathbb{Z}}

\numberwithin{equation}{section}

\newcommand{\ud}[0]{\,\mathrm{d}}

\newcommand{\dist}[0]{\operatorname{dist}}

\newcommand{\BMO}[0]{\operatorname{BMO}}

\newcommand{\calD}[0]{\mathcal{D}}
\newcommand{\rest}{{\lfloor}}

\swapnumbers
\theoremstyle{plain}
\newtheorem{thm}[equation]{Theorem}
\newtheorem{lem}[equation]{Lemma}
\newtheorem{prop}[equation]{Proposition}
\newtheorem{cor}[equation]{Corollary}

\theoremstyle{definition}
\newtheorem{defn}[equation]{Definition}

\newtheorem{exmp}[equation]{Example}

\theoremstyle{remark}
\newtheorem{rem}[equation]{Remark}

\title[A new approach to non-homogeneous local $Tb$ theorems]{A new approach to non-homogeneous local $Tb$ theorems}

\author{Henri Martikainen}
\address[H.M.]{Department of Mathematics and Statistics, University of Helsinki, P.O.B. 68, FI-00014 University of Helsinki, Finland}
\email{henri.martikainen@helsinki.fi}
\thanks{H.M. was supported by the Academy of Finland through the grant
Multiparameter dyadic harmonic analysis and probabilistic methods. Research of M.M. was supported by the ERC grant 320501 of the European Research Council (FP7/2007-2013).
M.M. was also supported  by IKERBASQUE and partially supported by the grant MTM-2017-82160-C2-2-P of the Ministerio de Econom\'ia y Competitividad (Spain). H.M. and E.V. are members of the Finnish Centre of Excellence in Analysis and Dynamics Research.}

\author{Mihalis Mourgoglou}
\address[M.M.]{Departament de Matem\`atiques, Universitat Aut\`onoma de Barcelona and Centre de Reserca Matem\` atica, Edifici C Facultat de Ci\`encies, 08193 Bellaterra (Barcelona)}
\curraddr{Departamento de Matem\'aticas, Universidad del Pa\' is Vasco, Barrio Sarriena s/n 48940 Leioa, Spain and\\
Ikerbasque, Basque Foundation for Science, Bilbao, Spain.}
\email{michail.mourgoglou@ehu.eus}

\author{Emil Vuorinen}
\address[E.V.]{Department of Mathematics and Statistics, University of Helsinki, P.O.B. 68, FI-00014 University of Helsinki, Finland}
\email{emil.vuorinen@helsinki.fi}

\makeatletter
\@namedef{subjclassname@2010}{%
  \textup{2010} Mathematics Subject Classification}
\makeatother


\begin{document}
\maketitle

\begin{abstract}
We develop a new general method to prove various non-doubling local $Tb$ theorems.
The method combines the non-homogeneous good lambda method of Tolsa,
the big pieces $Tb$ theorem of Nazarov--Treil--Volberg and a new change of measure argument
based on stopping time techniques. We also improve known results and discuss some further applications.
\end{abstract}
\section{Introduction}
By a \emph{singular integral operator} (SIO) related to some Borel measure $\mu$ in $\R^n$ we understand a linear operator $T$ for which
we have for all nice and disjointly supported functions $f_1, f_2$ that
$$
\int_{\R^n} Tf_1(x) f_2(x)\ud \mu(x) = \iint_{\R^n \times \R^n} K(x,y) f_1(y) f_2(x)\ud \mu(y) \ud \mu(x)
$$ for a suitable kernel $K$ defined for $x \ne y$.
An example in $\R$
is the Hilbert transform, where $\ud \mu(x) = \ud x$ and $K(x,y) = 1/(x-y)$.
In general, SIOs need not be bounded in $L^2(\mu)$:
the assumed size properties of the kernel $K$ are too weak for this and more subtle cancellation has to be accounted for.
This additional cancellation can most conveniently be formulated using various testing conditions.
Such theorems characterising the boundedness of SIOs via testing conditions are called $Tb$ theorems. The rough idea
is that the action of $T$ needs to be studied only on some single non-degenerate function $b$ (such as $b=1$), or in the case of local $Tb$ theorems,
on a suitable family of non-degenerate functions $b_Q$ indexed by cubes $Q$ (such as $b_Q = 1_Q$). An example of such a testing condition
is the local $T1$ condition:
$$
\sup_{\substack{Q \subset \R^n \\ Q \textup{ is a cube}}} \Big[ \frac{1}{\mu(Q)} \int_Q |T1_Q|^2\ud \mu + \frac{1}{\mu(Q)} \int_Q |T^{*}1_Q|^2\ud \mu \Big] < \infty.
$$
The theory is significantly more subtle if the underlying measure $\mu$ is non-doubling: such theory is called non-homogeneous Calder\'on--Zygmund analysis.
The usual assumption then is that $\mu(B(x,r)) \le Cr^m$ for some exponent $m$, and that the kernel estimates for $K$ are tied to this $m$ (see
Section \ref{sec:CZO}).

In this paper we develop a new method to prove certain very rough (meaning that the test functions $b_Q$ can have very low integrability)
local $Tb$ theorems, and prove new and very general results. The theorems are formulated in the
non-homogeneous situation. However, the results are already new even in the Lebesgue measure case.
The developed method is very convenient and modular, and we will refer to multiple further works where it has been successfully applied.

\subsection*{History and context}
The history of $Tb$ theorems is extremely wide. The starting point in the Lebesgue measure case is the famous $T1$ theorem by
David--Journ\'e \cite{DJ}. Our focus is, however, on the local variants that were first introduced by M. Christ \cite{Ch}. They are flexible tools
as the non-degeneracy (also referred to as accretivity)
of a given test function $b_Q$ is only assumed on its supporting cube $Q$, i.e., $|\int_Q b_Q\,d\mu| \gtrsim \mu(Q)$.
In a global $Tb$ theorem a single function $b$ has to satisfy this non-degeneracy in all cubes, and constructing such
a $b$ can be more difficult. Christ's version works for doubling measures $\mu$ (but not more general than that) and
requires that the test functions themselves are nice, $b_Q \in L^{\infty}(\mu)$, and that $Tb_Q$ satisfies a demanding testing condition, $Tb_Q \in L^{\infty}(\mu)$. 
Very significant efforts have been made by multiple authors
to allow \emph{both} rougher test functions and less demanding testing conditions. A parallel line of investigation
has dealt with the corresponding results in the non-homogeneous situation -- we get to this later.

In general, we actually require two families of test functions -- one for $T$ and one for the adjoint $T^{*}$. In what follows
we will always assume that $T$ is antisymmetric, $T^* = -T$, which makes the discussion easier to follow. However, the completely general case sometimes involves very real difficulties -- but we are not concerned with them here.

We say that a function $b_Q$ is an $L^p(\mu)$-admissible, $p \in [1, \infty)$, test function on a cube $Q \subset \R^n$ (with constant $B_1$), if
\begin{enumerate}
\item spt$\,b_Q \subset Q$, $\mu(Q) = \int_Q b_Q\,d\mu$ and
\item $\big( \frac{1}{\mu(Q)} \int_Q |b_Q|^p\,d\mu \big)^{1/p} \le B_1$.
\end{enumerate}
A long standing problem (even for the Lebesgue measure) asks whether the $L^2$ boundedness
of (an antisymmetric) SIO $T$ follows if we are given $p \in (1,\infty)$, and for every
cube $Q$ an $L^p(\mu)$-admissible test function $b_Q$ (with the same constant $B_1$) so that
\begin{equation}\label{eq:Hofmann}
\int_Q |Tb_Q|^{p'}\,d\mu \le B_2\mu(Q).
\end{equation}
Here $1/p + 1/p' = 1$.
This is often referred to as Hofmann's local $Tb$ problem.
For certain simpler model operators this type of local $Tb$ theorem in the Lebesgue measure case appears in Auscher--Hofmann--Muscalu--Tao--Thiele \cite{AHMTT}. The extension to the SIO case has turned out to be of extreme
difficulty -- particularly if $p < 2$. In fact, for $p < 2$ this is still not known in general. Even if the testing condition on $1_QTb_Q$ becomes more demanding here as $p'$ grows, i.e., when $p$ gets smaller, the lower integrability of the test functions $b_Q$ is desired.
Hyt\"onen--Nazarov \cite{HN} showed in the Lebesgue measure case that the $L^2$ boundedness follows from the \emph{buffered} testing condition
$
\int_{2Q} |Tb_Q|^{p'}\,dx \lesssim |Q|
$
for any $p \in (1,\infty)$.
A key thing in the Lebesgue measure case is that if $p \ge 2$, then the original testing conditions (with $1_QTb_Q$) automatically
imply the stronger buffered testing conditions (with $1_{2Q}Tb_Q$) by Hardy's inequality. Previous related results include Auscher--Yang \cite{AY}, Auscher--Routin \cite{AR} and Hofmann \cite{Ho1}. In the follow-up paper by two of us and Tolsa \cite{MMT} we use and build on the methods presented in this paper to fully prove the difficult case $p < 2$ in $\R$, and show
that we can always allow at least some exponents $p \in (2-\epsilon, 2]$ even when working in $\R^n$ with $n > 1$.

We turn to discuss the non-homogeneous aspects in more detail. Non-doubling theory has been pioneered by David \cite{Da1}, Nazarov--Treil--Volberg \cite{NTV} and Tolsa (see e.g. the book \cite{ToBook}). Previously, it was strongly believed that the class of doubling measures was the right class for the theory of singular integrals.
The non-homogeneous theory is of extreme importance in many geometric questions -- to mention just one key result, see the proof of the semiadditivity of analytic capacity by Tolsa \cite{ToBook, To1}. Moreover, Nazarov--Treil--Volberg \cite{NTV} developed their so-called dyadic-probabilistic methods to deal with the difficulties
posed by general measures, but such methods have also been extremely influential in other instances. For example, they led
to the dyadic representation theorems and the first solution of the $A_2$ conjecture by Hyt\"onen \cite{Hy}.

Nazarov--Treil--Volberg \cite{NTVa} proved the first local $Tb$ theorem in the non-doubling situation -- there $b_Q \in L^{\infty}(\mu)$ and
$Tb_Q \in \BMO_2(\mu)$ (understood to mean that for all cubes $R$ we have $\int_R |Tb_Q - \langle Tb_Q \rangle_R^{\mu}|^2\ud\mu \le C\mu(2R)$,
where $\langle Tb_Q \rangle^{\mu}_R = \mu(R)^{-1} \int_R Tb_Q \ud \mu$). In \cite{LM:CZO} one of us and Lacey proved a non-homogeneous local $Tb$ theorem
in the case that $b_Q, 1_QTb_Q \in L^{2}(\mu)$ (even for SIOs $T$ that are not necessarily antisymmetric). The state-of-the-art in the antisymmetric case is the already mentioned paper \cite{MMT}.

\subsection*{Square functions vs SIOs}
It is of interest to consider local $Tb$ theorems also for \emph{square functions (SFs)} $V$ that are introduced below. Such theorems have been applied e.g. in multiple Kato square root papers -- see for instance \cite{AHLMT}. See also Example \ref{ex:SFGEOM} below.
For us, however, SFs mainly just offer a simpler platform compared to SIOs, and allows us to present the full technical execution of our method, including the big piece $Tb$ type argument discussed in more detail below. A key technical difference between SFs $V$ and SIOs $T$ is that
for SIOs our testing conditions need to a priori involve $1_Q T_*b_Q$ (instead of $1_Q Tb_Q$), where $T_*$ is the maximally truncated SIO
\begin{displaymath}
T_*f(x) = \sup_{\epsilon > 0} |T_{\epsilon}f(x)|, \qquad T_{\epsilon}f(x) = \int_{|x-y| > \epsilon} K(x,y)f(y)\,d\mu(y).
\end{displaymath}
New local $Tb$ theorems for antisymmetric Calder\'on--Zygmund operators with testing conditions involving $1_QT_*b_Q$ follow at once from the ideas of this paper, and we explicitly state these in Section \ref{sec:CZO}.
To pass from such results to the original Hofmann's problem (with testing conditions involving $1_QTb_Q$ as in \eqref{eq:Hofmann})
a so-called adapted Cotlar's inequality is needed:
one needs to use the existence of the test functions to prove a Cotlar type inequality, i.e., a pointwise control of $T_*$ by something involving only $T$. A fancy version of an adapted Cotlar's inequality is presented in the follow-up paper \cite{MMT} (for previous
version see also \cite{HN}). This step causes the fact that in Hofmann's problem we can only allow $p \in (2-\epsilon, 2]$ (that is, we cannot
always go all the way down to $p > 1$ in \cite{MMT}).

Results involving $V$ and $T_*$, however, behave in completely analogous ways. In such formulations we can allow
$L^p(\mu)$-admissible test functions with any $p > 1$, and, as we shall soon see, we can do quite a bit more -- we can e.g. even use
test measures with weaker conditions than this. Moreover, the testing conditions need not be as in Hofmann's problem \eqref{eq:Hofmann}.
In fact, we can have $1_QVb_Q$ (or $1_QT_*b_Q$) in any $L^{s, \infty}$, $s > 0$. That is, the given regularity of the test function (the exponent $p$)
need not be reflected in the testing condition at all (we can use any $s > 0$) -- the conditions decouple.

We now define the SFs that we use.
Let $m, \alpha > 0$, and assume that we have kernels $s_{t}\colon \R^n \times \R^n \to \C$, $t > 0$, satisfying the size condition
\begin{equation}\label{eq:size}
|s_t(x,y)| \lesssim \frac{t^{\alpha}}{(t+|x-y|)^{m+\alpha}},
\end{equation}
the $y$-H\"older condition
\begin{equation}\label{eq:holy}
|s_t(x,y) - s_t(x,z)| \lesssim \frac{|y-z|^{\alpha}}{(t+|x-y|)^{m+\alpha}}
\end{equation}
whenever $|y-z| < t/2$, and the $x$-H\"older condition
\begin{equation}\label{eq:holx}
|s_t(x,y) - s_t(z,y)| \lesssim \frac{|x-z|^{\alpha}}{(t+|x-y|)^{m+\alpha}}
\end{equation}
whenever $|x-z| < t/2$.
Let $M(\R^n)$ denote the vector-space of all complex Borel measures in $\R^n$. The variation measure
of $\nu \in M(\R^n)$ is denoted by $|\nu|$ and the total variation is $\|\nu\| = |\nu|(\R^n)$.
For a given complex measure $\nu$ we define
\begin{displaymath}
\theta_t \nu(x) = \int s_t(x,y) \,d\nu(y), \qquad x \in \R^n.
\end{displaymath}
The vertical square function $V$ is defined by
\begin{displaymath}
V\nu(x) = \Big( \int_0^{\infty} |\theta_t \nu(x)|^2\,\frac{dt}{t} \Big)^{1/2}, \qquad x \in \R^n.
\end{displaymath}
Given a cube $Q \subset \R^n$ define also the truncated version
\begin{displaymath}
V_{Q}\nu(x) = \Big( \int_0^{\ell(Q)} |\theta_t \nu(x)|^2\,\frac{dt}{t} \Big)^{1/2},
\end{displaymath}
where $\ell(Q)$ denotes the side length of $Q$.

We say that a (positive) Radon measure $\mu$ in $\R^n$ is of order $m$, if $\mu(B(x,r)) \lesssim r^m$ for all $x \in \R^n$ and $r > 0$.
For $f \in \bigcup_{p \in [1,\infty]} L^p(\mu)$ and $x \in \R^n$ we set
\begin{displaymath}
\theta_t^{\mu} f(x) := \theta_t(f\,d\mu) = \int s_t(x,y) f(y)\,d\mu(y)
\end{displaymath}
and
\begin{displaymath}
V_{\mu}f(x) := V(f\,d\mu)(x) = \Big( \int_0^{\infty} |\theta_t^{\mu} f(x)|^2\,\frac{dt}{t} \Big)^{1/2}.
\end{displaymath}
Define also $V_{\mu, Q}f = V_Q(f\,d\mu)$. The above definitions make sense also when $\mu$ is finite (and not necessarily of order $m$).

\begin{exmp}\label{ex:SFGEOM}
Mayboroda--Volberg \cite{MV}, Chousionis, Garnett, Le and Tolsa \cite{CGLT} and others
have linked boundedness of square functions and geometry.
Let $E \subset \R^n$ be a closed set which is $m$-ADR for some integer $0 < m < n$, i.e., $\mu := \mathcal{H}^m |_E$ satisfies $\mu(B(x,r)) \sim r^m$
for all $x \in E$ and $r \in (0, \textup{diam}(E))$.
Let $s_t(x,y)=t\partial_t [t^{-m}\phi ([x-y]/t)]$, where $\phi(x)=(1+|x|^2)^{-(m+1)/2}$. This kernel satisfies  \eqref{eq:size}, \eqref{eq:holy} and \eqref{eq:holx}. One of the results of \cite{CGLT} says that $E$ is uniformly $m$-rectifiable (for the definition see \cite{CGLT}) if and only if $V_{\mu}$ is $L^2(\mu)$ bounded.
\end{exmp}

\subsection*{Description of the new method and applications}
In the first step of the method the idea is to fix a cube $Q$ and prove that there exists
$G_Q \subset Q$ so that $\mu(G_Q) \gtrsim \mu(Q)$ and $\| 1_{G_Q} V_{\mu} f\|_{L^2(\mu)} \lesssim \|f\|_{L^2(\mu)}$
for every $f \in L^2(\mu)$ satisfying that spt$\,f \subset G_Q$.
This relies on the big pieces type $Tb$ theorem of Nazarov, Treil and Volberg \cite{NTV:Vit} -- see also Volberg's book \cite{Vo} and Tolsa's book \cite{ToBook} for expositions of this theorem. In the appendix we also formulate and give a complete proof of the big piece $Tb$ theorem in the SF situation, which is much more approachable than the difficult one in the SIO context (compare e.g. to the book of Tolsa \cite{ToBook}, pages 137--194). We hope that our proof helps to make this important theorem more approachable. In our rough local $Tb$ setting this theorem certainly cannot be applied directly -- it e.g. requires a bounded test function.
The idea is to first perform a change of measure to $\sigma := |b_Q|\,d\mu$ and aim to apply the  big piece $Tb$ theorem \ref{thm:Tb}
with the measure $\sigma$ and the $L^{\infty}$-function $b_Q/|b_Q|$. In essence, we make the measure worse but the function a lot better.
In order to pass from the $\mu$ measure to the $\sigma$ measure, and then back, we use stopping time arguments.

The non-homogeneous good lambda method of Tolsa \cite{ToBook} -- see Theorem \ref{thm:goodlambda} below -- says that
the existence of such a big piece $G_Q$ in every nice cube (doubling and of small boundary) is enough to guarantee
the global $L^2(\mu)$ boundedness of $V_{\mu}$. It is extremely convenient that the good lambda method is so flexible -- this is the reason
why we can assume the existence of $b_Q$ only in nice cubes.

This general method is the main contribution of this article. Previously one of the key difficulties in proving non-homogeneous local $Tb$ theorems
with rough test functions was that it was difficult to prove the boundedness of certain $b_Q$-adapted (or twisted) martingale transforms (see Lacey--Martikainen \cite{LM:CZO}). We do not need such transformations in our new strategy, and this is highly convenient.

We compare our method to that of Hyt\"onen--Nazarov \cite{HN}. Their proof works for doubling measures
and it does not appear to be clear how to extend their method to non-homogeneous measures.
Moreover, they do not use the big pieces global $Tb$  or the the good lambda method. Instead, they write a rough test function as a sum of the good and bad part using the Calder\'on--Zygmund decomposition. 
The good part is a non-degenerate $L^{\infty}$ function which they want to use as a test function in some simpler $Tb$ theorem. To transfer their testing conditions involving $Tb_Q$ to this new bounded function they have to perform stopping times in a delicate order, and suppress their operator appropriately in the bad set. Thus, their proof is very different from our strategy. In \cite{MMT} we prove a deeper instance of Cotlar's inequality than in \cite{HN}, and then combine this with our local $Tb$ theorems involving $T_*$ (Proposition \ref{prop:TstarmainProp} below).
This combination then allows us to drop their buffer assumption even for some $p < 2$, and gives a proof that works in the non-homogeneous situation.

Besides the best known local $Tb$ theorem for SFs and maximally truncated SIOs $T_*$ presented in this paper and the best known result concerning Hofmann's problem \cite{MMT}, we have
also applied this new method in \cite{MMV}, where the flexibility of the method is helpful in handling issues involving metric space
arguments. Lastly, we mention that in \cite{MV} we applied the ideas of the current paper and \cite{MMT} to prove
new results in the bilinear setting.

\subsection*{Formulation of the new local $Tb$ theorem for square functions}
A cube $Q \subset \R^n$ is called $(\alpha,\beta)$-doubling for a given measure $\mu$ if $\mu(\alpha Q) \le \beta\mu(Q)$.
Given $t > 0$ we say that a cube $Q \subset \R^n$ has $t$-small boundary with respect to the measure $\mu$ if
\begin{displaymath}
\mu(\{x \in 2Q\colon\, \dist(x,\partial Q) \le \lambda\ell(Q)\}) \le t\lambda\mu(2Q)
\end{displaymath}
for every $\lambda > 0$.
The following theorem presents the best known local $Tb$ theorem for SFs.
\begin{thm}\label{thm:main}
Let $\mu$ be a measure of order $m$ in $\R^n$ and $B_1, B_2 < \infty$, $\epsilon_0 \in (0,1)$ be given constants.
Let $\beta > 0$ and $C_1$ be large enough (depending only on $n$). Suppose that for every
$(2,\beta)$-doubling cube $Q \subset \R^n$ with $C_1$-small boundary there exists a complex measure $\nu_Q$ so that
\begin{enumerate}
\item spt$\,\nu_Q \subset Q$;
\item $\mu(Q) = \nu_Q(Q)$;
\item $\|\nu_Q\| \le B_1 \mu(Q)$;
\item For all Borel sets $A \subset Q$ satisfying $\mu(A) \le \epsilon_0 \mu(Q)$ we have
$$
|\nu_Q|(A) \le \frac{\|\nu_Q\|}{32B_1}.
$$
\end{enumerate}
Suppose there exist $s > 0$ and for all $Q$ as above a Borel set $U_Q \subset \R^n$
such that $|\nu_Q|(U_Q) \le \frac{\|\nu_Q\|}{16B_1}$ and
\begin{displaymath}
\sup_{\lambda > 0} \lambda^s \mu(\{x \in Q \setminus U_Q\colon\, V_{Q}\nu_Q(x) > \lambda\}) \le B_2\|\nu_Q\|.
\end{displaymath}
Then $V_{\mu}\colon L^p(\mu) \to L^p(\mu)$ for every $p \in (1,\infty)$.
\end{thm}
The following points are aimed to clarify the technical aspects of this theorem.
\begin{itemize}
\item If $V_{\mu} \colon L^2(\mu) \to L^2(\mu)$ boundedly, then $V\colon M(\R^n) \to L^{1,\infty}(\mu)$ boundedly. 
In this non-homogeneous setting this is (for SIOs) a result of Nazarov--Treil--Volberg \cite{NTV:IMRN} -- for another proof see \cite{ToBook}. 
It follows that given $\nu_Q$ like above
one has to have
$$
\sup_{\lambda > 0} \lambda \mu(\{x \in Q\colon\, V_{Q}\nu_Q(x) > \lambda\}) \le \sup_{\lambda > 0} \lambda \mu(\{x\colon\, V\nu_Q(x) > \lambda\}) \le C\|\nu_Q\|.
$$
This makes the assumptions necessary.
\item Suppose $q \in (1,\infty)$ and that $b_Q$ is an $L^q(\mu)$-admissible test function with a constant $B_1$. Then $\nu_Q := b_Q\,d\mu$ is a testing measure
as in the above theorem -- i.e., it satisfies (1)-(4). Here $q=1$ is enough for (3) but $q > 1$ can be used to get (4):
$$
\int_A |b_Q|\ud\mu \le \mu(A)^{1/q'} B_1 \mu(Q)^{1/q} \le \epsilon_0^{1/q'} B_1 \int_Q b_Q\ud \mu
\le \frac{1}{32B_1} \int_Q |b_Q|\ud\mu
$$
if $\epsilon_0 := (32B_1^2)^{-q'}$. In general, it is enough to prove (4) by some method.
\item It is to be expected that the fact that the cubes are doubling and of small boundary makes the theorem significantly easier to apply with non-homogeneous measures than the previously known theorems. We are not aware how to modify the strategies in \cite{LM:CZO}, \cite{LM:SF} and \cite{MM2} to yield this generality.
\item The fact that we can allow a small exceptional set $U_Q$ allows some additional flexibility in technical arguments, and this is in fact exploited in \cite{MMT}.
\end{itemize}
In the previously known state-of-the-art theorems for SFs the assumptions have been of the form that $q \in (1,\infty)$, for each cube $Q$ we are given an
$L^q(\mu)$-admissible test function $b_Q$ and the testing condition holds with the same $q$:
$$
\int_Q |V_{\mu, Q} b_Q|^q\ud \mu \lesssim \mu(Q).
$$
The Lebesgue case is by Hofmann \cite{Ho} and the non-doubling case by two of us \cite{MM2}. Previous
results (the case $q=2$) include \cite{AHLMT} and \cite{LM:SF}.

In Section \ref{s:local} we prove our main result for SFs and
present our method in detail. In Section \ref{sec:CZO} we discuss the analogous results for SIOs. In Appendix \ref{s:global} we give a relatively short proof of the big pieces global $Tb$ theorem for square functions, which we need to apply in Section \ref{s:local}.

\subsection*{Notation}
We write $A \lesssim B$, if there is a constant $C>0$ so that $A \leq C B$. We may also write $A \sim B$ if $B \lesssim A \lesssim B$.

We then set some dyadic notation. For cubes $Q$ and $R$ we denote
\begin{itemize} 
\item $\ell(Q)$ is the side-length of $Q$;
\item $d(Q,R)$ denotes the distance between the cubes $Q$ and $R$;
\item $D(Q,R):=d(Q,R)+\ell(Q)+\ell(R)$ is the long distance;
\item $W_Q=Q\times [\ell(Q)/2, \ell(Q))$ is the Whitney region associated with $Q$;
\item $\text{ch}(Q)$ denotes the dyadic children of $Q$;
\item $\mu \rest Q$ denotes the measure $\mu$ restricted to $Q$;
\item $\langle f \rangle_Q^{\mu} = \mu(Q)^{-1}\int_Q f\,d\mu$ (or just $\langle f \rangle_Q$ if the measure is clear from the context).
\end{itemize}

\subsection*{Acknowledgements}
Much of this research was carried out when H.M. was visiting CRM during September 2015, and he would like to thank the institution for its hospitality. We thank Benjamin Jaye for answering our question regarding a technical point in some $Tb$ theorems. We would also like to thank
the anonymous referee and the editor for comments that helped to clarify the exposition.

\section{The main method and proof of the local $Tb$ theorem for SFs}\label{s:local}
We record the following easy lemma.
\begin{lem}\label{lem:sup}
Let a cube $Q \subset \R^n$ be given and $G \subset Q$. Suppose also that $\nu(Q) \lesssim \ell(Q)^m$.
If $\|1_GV_{\nu, Q} f\|_{L^2(\nu)} \lesssim \|f\|_{L^2(\nu)}$ for every $f \in L^2(\nu)$ satisfying spt$\,f \subset G$, then
also $\|1_GV_{\nu} f\|_{L^2(\nu)} \lesssim \|f\|_{L^2(\nu)}$ for every $f \in L^2(\nu)$ satisfying spt$\,f \subset G$.
\end{lem}
\begin{proof}
This follows from the pointwise estimate
\begin{displaymath}
\int_{\ell(Q)}^{\infty} |\theta_t^{\nu}f(x)|^2 \, \frac{dt}{t} \lesssim \nu(Q) \ell(Q)^{-2m} \|f\|_{L^2(\nu)}^2.
\end{displaymath}
\end{proof}
\begin{defn}\label{defn:random}
Given a cube $Q \subset \R^n$ we consider the following random dyadic grid.
For small notational convenience assume that $c_{Q} = 0$ (that is, $Q$ is centred at the origin). 
Let $N \in \Z$ be defined by the requirement $2^{N-3} \le \ell(Q) < 2^{N-2}$.
Consider the random square $Q^* = Q^*(w) = w + [-2^N,2^N)^n$, where
$w \in [-2^{N-1}, 2^{N-1})^n =: \Omega_N = \Omega$. The set $\Omega$ is equipped with the normalised Lebesgue measure $\mathbb{P}_N = \mathbb{P}$.
We define the grid $\mathcal{D}(w) := \mathcal{D}(Q^*(w))$. Notice that
$Q \subset \alpha Q^*(w)$ for some $\alpha < 1$, and $\ell(Q) \sim \ell(Q^*(w))$.
\end{defn}

Next, we prove the main Proposition.
\begin{prop}\label{prop:main}
Let $\mu$ be a measure of order $m$ and $B_1, B_2 < \infty$, $\epsilon_0 \in (0,1)$ be given constants. Let $Q \subset \R^n$ be a fixed cube. 
Assume that there exists a complex measure $\nu = \nu_Q$ such that
\begin{enumerate}
\item spt$\,\nu \subset Q$;
\item $\mu(Q) = \nu(Q)$;
\item $\|\nu\| \le B_1 \mu(Q)$;
\item For all Borel sets $A \subset Q$ satisfying $\mu(A) \le \epsilon_0 \mu(Q)$ we have
$$
|\nu|(A) \le \frac{\|\nu\|}{32B_1}.
$$
\end{enumerate}
Suppose there exist $s > 0$ and a Borel set $U_Q \subset \R^n$
for which $|\nu|(U_Q) \le \frac{\|\nu\|}{16B_1}$ so that
\begin{displaymath}
\sup_{\lambda > 0} \lambda^s \mu(\{x \in Q \setminus U_Q\colon\, V_{Q}\nu(x) > \lambda\}) \le B_2\|\nu\|.
\end{displaymath}
Then, there is some subset $G_Q \subset Q \setminus U_Q$ such that $\mu(G_Q) \gtrsim \mu(Q)$ and
\begin{displaymath}
\| 1_{G_Q} V_{\mu} f\|_{L^2(\mu)} \lesssim \|f\|_{L^2(\mu)}
\end{displaymath}
for every $f \in L^2(\mu)$ satisfying that spt$\,f \subset G_Q$.
\end{prop}
\begin{proof}
We can assume that spt$\, \mu \subset Q$. Indeed, if we have proved the theorem for such measures, we can then apply it to $\mu\rest Q$.
Let us denote $\sigma = |\nu|$, where $|\nu|$ is the variation measure of $\nu$. Also, let us write the polar decomposition of the complex measure $\nu$
as $\nu = b\,d\sigma$, where $b$ is a function so that $|b(x)| = 1$ always.

The idea is to apply the big pieces global $Tb$ theorem from Appendix \ref{s:global} (Theorem \ref{thm:Tb}).
It will be applied to the measure $\sigma$ and the bounded function $b$. Using stopping times
we need to construct some exceptional sets so that the assumptions of that theorem are verified. Moreover, we need to
be able to come back to the $\mu$ measure -- this requires encompassing additional stopping times to the construction.

We fix $w$, and write $\mathcal{D}(w) = \calD$. We also write $\calD_0 = \calD(0)$.
Let $\mathcal{A} = \mathcal{A}_w$ consist of the maximal dyadic cubes $R \in \calD$ for which
\begin{displaymath}
\Big| \int_R b \,d\sigma \Big| < \eta \sigma(R),
\end{displaymath}
where $\eta := \frac{1}{2}B_1^{-1}$.
We set
\begin{displaymath}
T = T_w = \bigcup_{R \in \mathcal{A}} R \subset \R^n.
\end{displaymath} 
Notice that
\begin{align*}
\sigma(Q) = \|\nu\| \le B_1\mu(Q) = B_1 \nu(Q) =  B_1 \int_Q b \,d\sigma.
\end{align*}
Then estimate
\begin{displaymath}
\int_Q b \,d\sigma = \Big| \int_Q b \,d\sigma \Big| = \Big|  \int_{Q \setminus T} b \,d\sigma + \sum_{R \in \mathcal{A}} \int_R  b \,d\sigma \Big|
 \le \sigma(Q \setminus T) + \eta\sigma(Q).
\end{displaymath}
Since $\eta B_1 = 1/2$ we conclude that
\begin{displaymath}
\sigma(Q) \le B_1 \sigma(Q \setminus T)  + \frac{1}{2}\sigma(Q),
\end{displaymath}
and so
\begin{displaymath}
\sigma(Q) \le 2B_1 [\sigma(Q) - \sigma(T)].
\end{displaymath}
From here we can read that
\begin{displaymath}
\sigma(T) \le (1-\eta) \sigma(Q).
\end{displaymath}

Next, let $\mathcal{F}$ consist of the maximal dyadic cubes $R \in \calD_0$ for which
\begin{displaymath}
\sigma(R) > \frac{B_1}{\epsilon_0} \mu(R)
\end{displaymath}
or
\begin{displaymath}
\sigma(R) < \delta \mu(R),
\end{displaymath}
where $\delta := \eta/16 = \frac{1}{32}B_1^{-1}$.
Let $\mathcal{F}_1$ be the collection of maximal cubes $R \in \calD_0$ satisfying the first condition, and define $\mathcal{F}_2$ analogously.
Note that
\begin{displaymath}
\mu\Big( \bigcup_{R \in \mathcal{F}_1} R \Big) \le \epsilon_0\mu(Q),
\end{displaymath}
so that we have by assumption (4) that
\begin{align*}
\sigma\Big( \bigcup_{R \in \mathcal{F}_1} R \Big) \le \frac{1}{32B_1} \sigma(Q) = \delta\sigma(Q).
\end{align*}
Finally, we record that
\begin{displaymath}
\sigma\Big( \bigcup_{R \in \mathcal{F}_2} R \Big) = \sum_{R \in \mathcal{F}_2} \sigma(R) \le \delta \sum_{R \in \mathcal{F}_2} \mu(R) = \delta\mu\Big( \bigcup_{R \in \mathcal{F}_2} R \Big)
\le \delta \mu(Q) \le \delta\sigma(Q).
\end{displaymath}
We may conclude that the set
\begin{displaymath}
H_1 = \bigcup_{R \in \mathcal{F}} R
\end{displaymath}
satisfies  $\sigma(H_1) \le 2\delta\sigma(Q) = \frac{\eta}{8}\sigma(Q)$.

We now record the important property of the exceptional set $H_1$. Let $x \in Q \setminus H_1$.
For any $R \in \calD_0$ satisfying that $x \in R$ we have that
\begin{align*}
\frac{1}{32B_1} = \delta \le \frac{\sigma(R)}{\mu(R)} \le \frac{B_1}{\epsilon_0}.
\end{align*}
From this we can conclude (using a dyadic variant of Lemma 2.13 of \cite{MaBook}) that for all Borel sets $A \subset \R^n$ there holds that
$$
\delta \mu(A \cap (Q \setminus H_1)) \le \sigma(A \cap (Q \setminus H_1)) \le \frac{B_1}{\epsilon_0} \mu(A \cap (Q\setminus H_1)).
$$
In particular, we have that $\sigma \rest (Q \setminus H_1) \ll \mu \rest (Q \setminus H_1)$. Using Radon--Nikodym theorem we let $\varphi \ge 0$ be a function so that
$$
\sigma(A) = \int_A \varphi \,d\mu
$$
for all Borel sets $A \subset Q \setminus H_1$. We obviously have that $\varphi \sim 1$ for $\mu$-a.e. $x \in Q \setminus H_1$.

We need another exceptional set $H_2$. To this end, let
\begin{displaymath}
p(x) := M_{R, m} \nu(x) =  \sup_{r > 0} \frac{\sigma(B(x,r))}{r^m}.
\end{displaymath}
For $p_0 > 0$ let $E_{p_0} = \{p \ge p_0\}$. Using that $M_{R, m} \colon M(\R^n) \to L^{1,\infty}(\mu)$ boundedly we see that
\begin{align*}
\mu(E_{p_0}) = \mu(\{M_{R, m} \nu \ge p_0\}) \le \frac{C}{p_0} \|\nu\| \le \frac{CB_1}{p_0}\mu(Q).
\end{align*}
We fix $p_0 \lesssim 1$ so large that $\mu(E_{p_0/2^m}) \le \epsilon_0\mu(Q)$, so that in particular
$\sigma(E_{p_0/2^m}) \le \frac{\eta}{8}\sigma(Q)$.
For $x \in \{p > p_0\}$ define
\begin{displaymath}
r(x) = \sup\{r > 0\colon \sigma(B(x,r)) > p_0r^m\},
\end{displaymath}
and then set
\begin{displaymath}
H_2 := \bigcup_{x \in \{p > p_0\}} B(x,r(x)).
\end{displaymath}
It is clear that every ball $B_r$ with $\sigma(B_r) > p_0r^m$ satisfies $B_r \subset H_2$.
Notice that if $y \in H_2$, then there is $x \in \{p > p_0\}$ so that $y \in B(x,r(x))$, and so
$\sigma(B(y, 2r(x)) \ge \sigma(B(x,r(x)) \ge p_0r(x)^m = p_02^{-m}[2r(x)]^m$. We conclude that
$H_2 \subset E_{p_0/2^m}$, and so $\sigma(H_2) \le \frac{\eta}{8}\sigma(Q)$. 


The assumption about the set $U_Q$ reads $\sigma(U_Q) \le \frac{\eta}{8}\sigma(Q)$.
Define now $H = H_1 \cup H_2 \cup U_Q$. The properties of $H$ are as follows:
\begin{enumerate}
\item We have $\sigma(H) \le \frac{\eta}{2}\sigma(Q)$, and so
$\sigma(H \cup T_w) \le (1-\eta/2)\sigma(Q) = \tau_1 \sigma(Q)$, $\tau_1 < 1$.
\item If $\sigma(B_r) > p_0r^m$, then $B_r \subset H$.
\item We have a function $\varphi$ so that
$$
\sigma(A) = \int_A \varphi\,d\mu
$$
for all Borel sets $A \subset Q \setminus H$, and $\varphi \sim 1$ for $\mu$-a.e. $x \in Q \setminus H$.
\end{enumerate}
We also have for every $\lambda > 0$ that
\begin{align*}
\lambda^s & \sigma(\{x \in Q \setminus H\colon\,  V_{\sigma, Q} b(x) > \lambda\}) \\
&=\lambda^s \sigma(\{x \in Q \setminus H\colon\,  V_{Q} \nu(x) > \lambda\}) \\
&\lesssim \lambda^s\mu(\{x \in Q \setminus U_Q\colon\,  V_{Q} \nu(x) > \lambda\}) \le B_2\|\nu\| = B_2\sigma(Q).
\end{align*}
Appealing to Theorem \ref{thm:Tb} with the measure $\sigma$ and the $L^{\infty}$ function $b$ we find $G_Q \subset Q \setminus H \subset Q \setminus U_Q$ so that
$\sigma(G_Q) \gtrsim \sigma(Q)$ and
\begin{equation}\label{eq:sigmab}
\|1_{G_Q} V_{\sigma, Q}f \|_{L^2(\sigma)} \lesssim \|f \|_{L^2(\sigma)}
\end{equation}
for every $f \in L^2(\sigma)$.

Suppose now that $g \in L^2(\mu)$ and spt$\,g \subset G_Q$. We apply Equation \eqref{eq:sigmab} with $f = g / \varphi$ (since
$G_Q \subset Q \setminus H$ we have $\varphi \sim 1$ $\mu$-a.e. on the support of $g$).
Notice that
\begin{displaymath}
\|1_{G_Q} V_{\sigma, Q}(g/\varphi) \|_{L^2(\sigma)} = \|1_{G_Q} V_{\mu, Q}g \|_{L^2(\sigma)} \gtrsim \|1_{G_Q} V_{\mu, Q}g \|_{L^2(\mu)}
\end{displaymath}
so that
\begin{displaymath}
\|1_{G_Q} V_{\mu, Q}g \|_{L^2(\mu)} \lesssim \|g/\varphi \|_{L^2(\sigma)} \lesssim \|g\|_{L^2(\mu)}.
\end{displaymath}
Applying Lemma \ref{lem:sup} we conclude that
\begin{displaymath}
\|1_{G_Q} V_{\mu}f \|_{L^2(\mu)} \lesssim \|f\|_{L^2(\mu)}
\end{displaymath}
for every $f \in L^2(\mu)$ satisfying that spt$\,f \subset G_Q$. Moreover, we have that
\begin{displaymath}
\mu(Q) \le \sigma(Q) \lesssim \sigma(G_Q) = \int_{G_Q} \varphi \,d\mu \lesssim \mu(G_Q).
\end{displaymath}
We are done.
\end{proof}

Let us now record the non-homogenous good lambda method of Tolsa. 
This is essentially Theorem 2.22 in \cite{ToBook}, where it is done for Cald\'eron-Zygmund operators and without the requirement that the cubes in the statement should have small boundaries. The modified version with small boundaries is recorded in \cite{MMT}. The adaptation to the square function setting is easy, and we omit it.
\begin{thm}\label{thm:goodlambda}
Let $\mu$ be a measure of order $m$ in $\R^n$.
Let $\beta>0$ and $C_1 > 0$ be big enough numbers, depending only on the dimension $n$, and assume $\theta \in (0,1)$. Suppose for each $(2,\beta)$-doubling cube $Q$
with $C_1$-small boundary
there exists a subset $G_Q \subset Q$  such that
$\mu(G_Q)\geq \theta \mu(Q)$ and $V\colon M(\R^n) \to L^{1,\infty}(\mu \rest G_Q)$ is bounded with a uniform constant independent of $Q$.
Then $V_{\mu}$ is bounded in $L^p(\mu)$ for all $1<p<\infty$ with a constant depending on $p$ and on the preceding constants.
\end{thm}
\begin{rem}\label{rem:goodlambda}
In Theorem \ref{thm:goodlambda} the assumption $V\colon M(\R^n) \to L^{1,\infty}(\mu\rest G_Q)$ can be replaced by $V_{\mu}\colon L^2(\mu\rest G_Q) \to L^2(\mu\rest G_Q)$.
Indeed, the latter assumption implies the former one. The original reference is Nazarov--Treil--Volberg \cite{NTV:IMRN}, but
see also Theorem 2.16 in \cite{ToBook}.
\end{rem}

We are ready to prove our main theorem for SFs.
\begin{proof}[Proof of Theorem \ref{thm:main}]
Proposition \ref{prop:main} gives for every $(2,\beta)$-doubling cube $Q \subset \R^n$ with $C_1$-small boundary a subset
$G_Q \subset Q$ such that $\mu(G_Q) \gtrsim \mu(Q)$ and
\begin{displaymath}
\| 1_{G_Q} V_{\mu} f\|_{L^2(\mu)} \lesssim \|f\|_{L^2(\mu)}
\end{displaymath}
for every $f \in L^2(\mu)$ with spt$\,f \subset G_Q$.
Applying Theorem \ref{thm:goodlambda} and Remark \ref{rem:goodlambda} gives the result.
\end{proof}

\section{Comments on the Calder\'on--Zygmund case}\label{sec:CZO}
In this section we describe more carefully what kind of result can directly be obtained for Calder\'on--Zygmund operators (this means the same thing
as SIOs here) using the method of this paper. That is, we describe the analog of our main Proposition \ref{prop:main} for Calder\'on--Zygmund operators.
For clarity we formulate a slightly less technical statement involving testing functions rather than measures -- this is the one invoked in \cite{MMT}. First, let us recall some definitions.

We say that $K \colon \R^n \times \R^n \setminus \{(x,y): x = y\} \to \C$ is an $m$-dimensional Calder\'on--Zygmund kernel if
for some $C < \infty$ and $\alpha \in (0,1]$ we have that
\begin{displaymath}
|K(x,y)| \le \frac{C}{|x-y|^m}, \qquad x \ne y,
\end{displaymath}
\begin{displaymath}
|K(x,y) - K(x',y)| \le C\frac{|x-x'|^{\alpha}}{|x-y|^{m+\alpha}}, \qquad |x-y| \ge 2|x-x'|,
\end{displaymath}
and
\begin{displaymath}
|K(x,y) - K(x,y')| \le C\frac{|y-y'|^{\alpha}}{|x-y|^{m+\alpha}}, \qquad |x-y| \ge 2|y-y'|.
\end{displaymath}
We consider the following $\epsilon$-truncated singular integral operators $T_{\epsilon}$, $\epsilon > 0$:
\begin{displaymath}
T_{\epsilon}\nu(x) = \int_{|x-y| > \epsilon} K(x,y)\,d\nu(y), \qquad x \in \R^n.
\end{displaymath}
The integral on the right hand side is absolutely convergent if, say, $|\nu|(\R^n) < \infty$.

For a positive Radon measure $\mu$ in $\R^n$ and $f \in L^1_{\textup{loc}}(\mu)$ we define
\begin{displaymath}
T_{\mu, \epsilon}f(x)  = T_{\epsilon}(f\mu)(x), \qquad x \in \R^n.
\end{displaymath}
The integral defining $T_{\mu, \epsilon}f(x)$ is absolutely convergent if for example $f \in L^p(\mu)$ for some $1 \le p < \infty$ and $\mu$ is of order $m$.
We say that $T_{\mu}$ is bounded in $L^p(\mu)$ if the operators $T_{\mu, \epsilon}$ are bounded in $L^p(\mu)$ uniformly in $\epsilon > 0$.

We define the maximal operator $T_*$ by
\begin{displaymath}
T_{*}\nu(x) = \sup_{\epsilon > 0} |T_{\epsilon}\nu(x)|, \qquad \nu \in M(\R^n), \, x \in \R^n,
\end{displaymath}
Like above, we also set $T_{\mu, *}f(x) = T_*(f\mu)$.

For $p \in [1, \infty)$ we say that a function $b_Q$ is an $L^p(\mu)$-admissible test function on a cube $Q \subset \R^n$ (with constant $B_1$), if
\begin{enumerate}
\item spt$\,b_Q \subset Q$,
\item $\mu(Q) = \int_Q b_Q\,d\mu$ and
\item $\big( \frac{1}{\mu(Q)} \int_Q |b_Q|^p\,d\mu \big)^{1/p} \le B_1$.
\end{enumerate}
Below we will need $p > 1$ (see the comments after Theorem \ref{thm:main}).

It is easier to prove local $Tb$ theorems assuming conditions for maximal truncations $T_{\mu, *}b_Q$ rather than uniform conditions on $T_{\mu, \epsilon}b_Q$.
Of course, this distinction does not manifest itself in the square function setting.
There is a tradeoff here and the theorems are not strictly comparable. This is because
one needs much weaker conditions on $T_{\mu, *}b_Q$ compared to $T_{\mu, \epsilon}b_Q$, but of course $T_{\mu, *}b_Q$ is a larger
object to begin with. A theorem involving $T_{\mu, *}b_Q$ with very weak testing assumptions is very important also because
the most efficient strategies for proving local $Tb$ theorems involving $T_{\mu, \epsilon}b_Q$ are based on these. See Hyt\"onen--Nazarov \cite{HN}
for some related results in the Lebesgue situation. 

We now state the analog of Proposition \ref{prop:main}  -- the result concerning Hofmann's problem \cite{MMT} is reduced to this statement
using a fancy version of Cotlar's inequality (the main technical tool of \cite{MMT}). Even the flexibility with the exceptional set $U_Q$ is needed there.

\begin{prop}\label{prop:TstarmainProp}
Let $\mu$ be a measure of degree $m$ on $\R^n$ and $K$ be an $m$-dimensional kernel satisfying $K(x,y) = -K(y,x)$.
Let $Q \subset \R^n$ be a fixed cube, $q \in (1,\infty)$ and $b_Q$ be an $L^q(\mu)$-admissible test function in $Q$ with constant $B_1$.
Then there exists a small constant $c_1 = c_1(q, B_1) > 0$ with the following property. If there exist $s > 0$ and an exceptional set $U_Q \subset \R^n$
so that $\int_{U_Q} |b_Q|\,d\mu \le c_1\int_Q |b_Q|\,d\mu$ and
\begin{equation*}
\sup_{\lambda > 0} \lambda^s \mu(\{x \in Q \setminus U_Q\colon\, T_{\mu,*}b_Q(x) > \lambda\}) \le B_2\mu(Q) \textup{ for some } B_2 < \infty,
\end{equation*}
then there exists $G_Q \subset Q \setminus E_Q$ so that $\mu(G_Q) \gtrsim \mu(Q)$ and
$T_{\mu\rest G_Q}\colon L^2(\mu\rest G_Q) \to L^2(\mu\rest G_Q)$ with a norm depending on the constants in the assumptions.
\end{prop}
The proof is essentially the same than that of Proposition \ref{prop:main}. However, this time we of course need to use the Calder\'on--Zygmund version of the big piece $Tb$ theorem due to Nazarov--Treil--Volberg \cite{NTV:Vit}. This is also proved in detail in \cite{Vo}. See also
\cite{ToBook}, Theorem 5.1, for an exposition in the case of the Cauchy operator.
Notice that the big piece $Tb$ involves maximal truncations as well, which is an explanation why they appear here also.

The related local $Tb$ corollary follows using the good lambda method. Again, we could even use testing measures, exceptional sets and so forth,
but we prefer to state the slightly less technical statement here.
\begin{cor}\label{cor:locTbTstar}
Let $\mu$ be a measure of degree $m$ on $\R^n$ and $K$ be an $m$-dimensional kernel satisfying $K(x,y) = -K(y,x)$.
Suppose $q \in (1,\infty)$, and let $b$ and $t$ be large enough constants (depending only on $n$).
We assume that to every $(5,b)$-doubling cube $Q \subset \R^n$ with $t$-small boundary there is associated
an $L^q(\mu)$-admissible test function $b_Q$ in $Q$ with constant $B_1$ such that
\begin{displaymath}
\sup_{\lambda > 0} \lambda^s \mu(\{x \in Q\colon\, T_{\mu,*}b_Q(x) > \lambda\}) \le B_2\mu(Q) \textup{ for some } B_2 < \infty \textup{ and } s > 0.
\end{displaymath}
Then $T_{\mu}\colon L^2(\mu) \to L^2(\mu)$ with a bound depending on the above constants.
\end{cor}

\appendix
\section{Big pieces global $Tb$ for square functions}\label{s:global}
In this appendix we prove the big pieces global $Tb$ theorem for square functions which we needed above. For antisymmetric
Calder\'on--Zygmund operators with some assumptions about the maximal truncation $T_{*}b$ this is by Nazarov--Treil--Volberg \cite{NTV:Vit} (see
also \cite{ToBook} and \cite{Vo}). Our efficient proof in the square function setting is much more approachable than the proof in the Calder\'on--Zygmund case, which is why we
present it here.
\begin{thm}\label{thm:Tb}
Let $Q \subset \R^n$ be a cube.
Let $\sigma$ be a finite Borel measure in $\R^n$ so that spt$\,\sigma \subset Q$.
Suppose $b$ is a function satisfying that $\|b\|_{L^{\infty}(\sigma)} \le C_b$.
For every $w$ let $T_w$ be the union of the maximal
dyadic cubes $R \in \calD(w)$ for which
\begin{displaymath}
\Big| \int_R b \,d\sigma \Big| < c_{\textup{acc}}\sigma(R).
\end{displaymath}
We are also given a measurable set $H \subset \R^n$ satisfying the following properties.
\begin{itemize}
\item There is $\delta_0 < 1$ so that $\sigma(H \cup T_w) \le \delta_0\sigma(Q)$ for every $w$.
\item Every ball $B_r$ of radius $r$ satisfying $\sigma(B_r) > C_0r^m$ satisfies $B_r \subset H$.
\item We have for some $s > 0$ the estimate
\begin{displaymath}
\sup_{\lambda > 0} \lambda^s \sigma(\{x \in Q \setminus H\colon\, V_{\sigma, Q}b(x) > \lambda\}) \le C_1\sigma(Q).
\end{displaymath}
\end{itemize}
Then there is a measurable set $G_Q$ satisfying $G_Q \subset Q \setminus H$ and the following properties:
\begin{enumerate}[(a)]
\item $\sigma(G_Q) \gtrsim \sigma(Q)$.
\item $\| 1_{G_Q} V_{\sigma, Q}f \|_{L^2(\sigma)} \lesssim \|f\|_{L^2(\sigma)}$ for every $f \in L^2(\sigma)$.
\end{enumerate}
\end{thm}
\begin{rem}
Only the good lambda method (Theorem \ref{thm:goodlambda}) and hence the main theorem (Theorem \ref{thm:main}) require the x-continuity of
$s_t$ i.e. \eqref{eq:holx}. Proposition \ref{prop:main} and the above Theorem \ref{thm:Tb} do not require it.
\end{rem}

\begin{proof}[Proof of Theorem \ref{thm:Tb}]
We begin by suppressing our operator appropriately.
Set
\begin{displaymath}
S_0 = \{x \in Q\colon\, V_{\sigma, Q}b(x) > \lambda_0\},
\end{displaymath}
where $0 < \lambda_0 \lesssim 1$ is large enough.
Now simply define
\begin{displaymath}
\tilde s_t(x,y) = s_t(x,y)1_{\R^n \setminus S_0}(x).
\end{displaymath}
Notice that $(\tilde s_t)_{t > 0}$ is a measurable family of kernels satisfying \eqref{eq:size} and \eqref{eq:holy}, which is all we shall need in what follows.
Now, $\widetilde V_{\sigma, Q}$ (and similar objects) are defined in the natural way using the kernels $\tilde s_t$. Then for any $f$ we have
\begin{equation}\label{formula for suppressed}
\widetilde V_{\sigma, Q}f(x) = V_{\sigma, Q}f(x) 1_{\R^n \setminus S_0}(x)=V_{\sigma, Q}f(x) 1_{\R^n \setminus(Q \cap \{V_{\sigma, Q}b > \lambda_0\})}(x),
\end{equation}
and from here we can easily read two key things about these suppressed operators. The first is that for any $f$ we have
\begin{equation}\label{eq:tildeVEqV}
\widetilde V_{\sigma, Q}f(x) =  V_{\sigma, Q}f(x) \textup{ for } x \in \R^n \setminus S_0,
\end{equation}
and the second is that
\begin{equation}\label{eq:tildeVb}
\widetilde V_{\sigma, Q}b(x) \le \lambda_0 \textup{ for every } x \in Q.
\end{equation}

Finally, with a large enough choice of $\lambda_0$ we have (for every $w$) that $\sigma(H \cup T_w \cup S_0) \le \delta_1\sigma(Q)$
for some $\delta_1 < 1$. Indeed, 
\begin{displaymath}
\sigma(S_0 \setminus H) \le \sigma(\{x \in Q \setminus H\colon\, V_{\sigma, Q}b(x) > \lambda_0\}) \le \frac{C_1}{\lambda_0^s} \sigma(Q).
\end{displaymath}
At this point $\lambda_0 \lesssim 1$ can be fixed by demanding that it satisfies
\begin{displaymath}
\lambda_0^s > \frac{2C_1}{1-\delta_0},
\end{displaymath}
whence we conclude that
\begin{equation}\label{eq:expmes}
\sigma(H \cup T_w \cup S_0) \le \sigma(H \cup T_w) + \sigma(S_0 \setminus H) \le \frac{1+\delta_0}{2}\sigma(Q) =: \delta_1\sigma(Q), \, \delta_1 < 1.
\end{equation}
We are now done with suppressing the operator.

We will next define the set $G_Q$. This is done by setting
\begin{displaymath}
p_0(x) = \mathbb{P}(\{w \in \Omega\colon\, x \in Q \setminus [H \cup T_w \cup S_0]\}),
\end{displaymath}
and then defining
\begin{displaymath}
G_Q = \Big\{x \in Q\colon\, p_0(x) > \frac{1-\delta_1}{2} =: \tau\Big\} \subset Q \setminus H.
\end{displaymath}
An argument by Nazarov--Treil--Volberg (see \cite{NTV:Vit}) shows that $\sigma(G_Q) \gtrsim \sigma(Q)$. Indeed, the argument goes as follows.
Notice first that by \eqref{eq:expmes} we have that
\begin{displaymath}
\int_Q p_0(x)\,d\sigma(x) = \int_{\Omega} \sigma(Q \setminus [H \cup T_w \cup S_0])\,d\mathbb{P}(w) \ge (1-\delta_1)\sigma(Q).
\end{displaymath}
Since $1-p_0 \ge 0$ everywhere, and $1-p_0 \ge 1-\tau = (1+\delta_1)/2$ on $Q \setminus G_Q$, we have
\begin{displaymath}
\int_Q (1-p_0(x))\,d\sigma(x) \ge \int_{Q \setminus G_Q} (1-p_0(x))\,d\sigma(x) \ge \frac{1+\delta_1}{2}\sigma(Q \setminus G_Q).
\end{displaymath}
We conclude that
\begin{displaymath}
\sigma(Q \setminus G_Q) \le \frac{2}{1+\delta_1}\Big( \sigma(Q) - \int_Q p_0(x)\,d\sigma(x) \Big) \le \frac{2\delta_1}{1+\delta_1}\sigma(Q),
\end{displaymath}
and so
\begin{displaymath}
\sigma(G_Q) \ge \Big(1-\frac{2\delta_1}{1+\delta_1}\Big)\sigma(Q) = \frac{1-\delta_1}{1+\delta_1}\sigma(Q).
\end{displaymath}

It remains to prove that $\| 1_{G_Q} V_{\sigma, Q}f \|_{L^2(\sigma)} \lesssim \|f\|_{L^2(\sigma)}$ for every $f \in L^2(\sigma)$.
The key property of $G_Q$ is as follows. Suppose $h \ge 0$ is any positive function. Then we have that
\begin{displaymath}
\int_{G_Q} h(x) \,d\sigma(x) \le \tau^{-1} \int_{G_Q} p_0(x)h(x)\,d\sigma(x) = \tau^{-1} E_w \int_{G_Q \setminus [H \cup T_w \cup S]} h(x)\,d\sigma(x).
\end{displaymath}
We apply this as follows:
\begin{align*}
\| 1_{G_Q} V_{\sigma, Q}f \|_{L^2(\sigma)}^2 &= \int_{G_Q} \int_0^{\ell(Q)} |\theta_t^{\sigma} f(x)|^2\,\frac{dt}{t}\,d\sigma(x)  \\
&\le \tau^{-1} E_w \int_{G_Q \setminus [H \cup T_w \cup S_0]} \int_0^{\ell(Q)} |\theta_t^{\sigma} f(x)|^2\,\frac{dt}{t}\,d\sigma(x) \\
&= \tau^{-1} E_w \sum_{R \in \calD_0}  \int_{[R \cap G_Q] \setminus [H \cup T_w \cup S_0]} \int_{\ell(R)/2}^{\min(\ell(R), \ell(Q))} |\theta_t^{\sigma} f(x)|^2\,\frac{dt}{t}\,d\sigma(x),
\end{align*}
where again $\calD_0 = \calD(0)$. Given $w$ we then write
\begin{displaymath}
\sum_{R \in \calD_0} = \mathop{\sum_{R \in \calD_0}}_{R \textup{ is } \calD(w)\textup{-good}} + \mathop{\sum_{R \in \calD_0}}_{R \textup{ is } \calD(w)\textup{-bad}},
\end{displaymath}
where $R \in \calD_0$ is said to be $\calD(w)$-good if $d(R, \partial P) > \ell(R)^{\gamma}\ell(P)^{1-\gamma}$ for every $P \in \calD(w)$ satisfying $\ell(P) \ge 2^r\ell(R)$.
Here $r \lesssim 1$ is a fixed large enough parameter, and $\gamma := \alpha/(2m+2\alpha)$. It is a standard fact by Nazarov--Treil--Volberg (see \cite{NTV})
that given $R \in \calD_0$ we have that
\begin{equation}\label{eq:bad}
\mathbb{P}(\{w \in \Omega\colon\, R\, \textup{ is } \calD(w)\textup{-bad}\}) \le \tau/2
\end{equation}
for a large enough fixed $r$.

Using \eqref{eq:bad} we estimate
\begin{align*}
E_w &\mathop{\sum_{R \in \calD_0}}_{R \textup{ is } \calD(w)\textup{-bad}} \int_{[R \cap G_Q] \setminus [H \cup T_w \cup S_0]} \int_{\ell(R)/2}^{\min(\ell(R), \ell(Q))} |\theta_t^{\sigma} f(x)|^2\,\frac{dt}{t}\,d\sigma(x) \\
&\le \mathop{\sum_{R \in \calD_0}} \mathbb{P}(\{w \in \Omega\colon\, R\, \textup{ is } \calD(w)\textup{-bad}\})  \int_{R \cap G_Q} \int_{\ell(R)/2}^{\min(\ell(R), \ell(Q))} |\theta_t^{\sigma} f(x)|^2\,\frac{dt}{t}\,d\sigma(x) \\
&\le \frac{\tau}{2} \int_{G_Q} \int_0^{\ell(Q)} |\theta_t^{\sigma} f(x)|^2\,\frac{dt}{t}\,d\sigma(x).
\end{align*}
To be precise, for the following we would need the a priori finiteness of this term. However, this is easy to arrange in a multiple of ways, so we skip this technicality.
We may now conclude (using also that $\theta^{\sigma}_t f(x) = \widetilde \theta^{\sigma}_t f(x)$ for every $x \in Q \setminus S_0$ by \eqref{eq:tildeVEqV}) that
\begin{align*}
\| 1_{G_Q}& V_{\sigma, Q}f \|_{L^2(\sigma)}^2 \\
&\le 2\tau^{-1} E_w  \mathop{\sum_{R \in \calD_0}}_{R \textup{ is } \calD(w)\textup{-good}}
\int_{[R \cap G_Q] \setminus [H \cup T_w \cup S_0]} \int_{\ell(R)/2}^{\min(\ell(R), \ell(Q))} |\widetilde \theta_t^{\sigma} f(x)|^2\,\frac{dt}{t}\,d\sigma(x) \\
&\lesssim  E_w \mathop{\mathop{\sum_{R \in \calD_0}}_{R \textup{ is } \calD(w)\textup{-good}}}_{R \not \subset H \cup T_w} \int_R \int_{\ell(R)/2}^{\min(\ell(R), \ell(Q))}
|\widetilde \theta_t^{\sigma} f(x)|^2\,\frac{dt}{t}\,d\sigma(x).
\end{align*}
We will now fix $w$, write $\calD = \calD(w)$ and $T = T_w$, and prove that
\begin{equation}\label{eq:mainthing}
\mathop{\mathop{\sum_{R \in \calD_0}}_{R \textup{ is } \calD\textup{-good}}}_{R \not \subset H \cup T} \int_R \int_{\ell(R)/2}^{\min(\ell(R), \ell(Q))}
|\widetilde \theta_t^{\sigma} f(x)|^2\,\frac{dt}{t}\,d\sigma(x) \lesssim \|f\|_{L^2(\sigma)}^2.
\end{equation}
This will then end the proof.

The important property of the set $T$ is that if $R \in \calD$ and
$R \not \subset T$ then
\begin{displaymath}
\Big| \int_R b\,d\sigma \Big| \gtrsim \sigma(R),
\end{displaymath}
while the important property of the set $H$ is that if $L \subset \R^n$ is an arbitrary cube satisfying $L \not \subset H$ then $\sigma(\lambda L) \lesssim \lambda^m \ell(L)^m$
for all $\lambda \ge 1$. It is useful to say that
$R \in \dtr_0$ (tr stands for \emph{transit}) if $R \in \calD_0$, $\sigma(R) \ne 0$ and $R \not \subset H \cup T$, and $P \in \dtr$ if $P \in \calD$, $\sigma(P) \ne 0$ and $P \not \subset H \cup T$.
Note that $\dtr_0$ really means $w$-transit cubes from $\calD_0$ (and one should really write $\dtr_0(w)$), but $w$ is fixed and so $T$ is fixed and we do not need to insist on this.

It is time to expand the function $f$ in the grid $\calD$ using $b$-adapted martingales only in the transit cubes $P \in \dtr$.
Denote $\langle f \rangle_A = \langle f \rangle_A^{\sigma} =  \sigma(A)^{-1} \int_A f\,d\sigma$, if $\sigma(A) \ne 0$. Let $P_0 = Q^*(w)$ (see the Definition \ref{defn:random})
so that all $P \in \calD$ satisfy $P \subset P_0$.
Without loss of generality we can assume that spt$\,b \subset Q$ and spt$\,f \subset Q$.
Define
\begin{displaymath}
E_{P_0}f = \frac{\langle f \rangle_{P_0}}{\langle b \rangle_{P_0}} b.
\end{displaymath}
(This is actually independent of $w$ since it just equals $E_Qf$, because spt$\,\sigma \subset Q \subset P_0$).
For any cube $P  \in \dtr$ define the function $\Delta_P f$ as follows:
\begin{displaymath}
\Delta_P f = \mathop{\sum_{P' \in \textup{ch}(P)}}_{\sigma(P') \ne 0} A_{P'}(f) 1_{P'},
\end{displaymath}
where
\begin{displaymath}
A_{P'}(f) = \left\{ \begin{array}{ll}
\Big( \frac{\langle f \rangle_{P'}}{\langle b \rangle_{P'}}-\frac{\langle f \rangle_P}{\langle b \rangle_P} \Big)b & \textrm{if } P' \in \dtr,\\
f-\frac{\langle f \rangle_P}{\langle b \rangle_P}b & \textrm{if } P' \not \in \dtr.\\
\end{array} \right.
\end{displaymath}
Notice that $P_0 \in \dtr$, since $\sigma(P_0) = \sigma(Q)$ and every non-transit cube $P$ has to satisfy $\sigma(P) \le \sigma(H \cup T) \le \delta_0\sigma(Q)$.
It is easy to see that
\begin{displaymath}
f = \sum_{P \in \dtr} \Delta_P f + E_{P_0}f
\end{displaymath}
$\sigma$-a.e. and in $L^2(\sigma)$, and that
\begin{displaymath}
\sum_{P \in \dtr} \|\Delta_P f\|_{L^2(\sigma)}^2 + \|E_{P_0}f\|_{L^2(\sigma)} \lesssim \|f\|_{L^2(\sigma)}^2.
\end{displaymath}
See e.g. Section 5.4.4 of \cite{ToBook}. It will be convenient to exploit notation by redefining on the largest level $P_0$ the operator $\Delta_{P_0}f$ to be $\Delta_{P_0}f + E_{P_0}f$.

Going back to \eqref{eq:mainthing} we see that we need to control
\begin{displaymath}
\mathop{\sum_{R \in \dtr_0}}_{R \textup{ is } \calD\textup{-good}}\int_R \int_{\ell(R)/2}^{\min(\ell(R), \ell(Q))}
\Big|\sum_{P \in \dtr} \widetilde \theta_t^{\sigma} \Delta_Pf(x)\Big|^2\,\frac{dt}{t}\,d\sigma(x).
\end{displaymath}
Given $R \in \dtr_0$, $R$ is $\calD(w)$-good, the $P \in \dtr$ summation is split in to the following four pieces:
\begin{enumerate}
\item $P$: $\ell(P)< \ell(R)$;
\item $P$: $\ell(P)\geq \ell(R)$ and $d(P,R)> \ell(R)^\gamma \ell(P)^{1-\gamma}$;
\item $P$: $\ell(R) \leq \ell(P) \leq 2^r \ell(R)$ and $d(P,R) \leq \ell(R)^\gamma \ell(P)^{1-\gamma}$;
\item $P$: $\ell(P)>2^r\ell(R)$ and $d(P,R) \leq \ell(R)^\gamma \ell(P)^{1-\gamma}$.
\end{enumerate}
For future need we set
\begin{align*}
A_{PR} &:= \frac{\ell(P)^{\alpha/2}\ell(R)^{\alpha/2}}{D(P,R)^{m+\alpha}} \sigma(P)^{1/2}\sigma(R)^{1/2}; \\
D(P,R) &:= \ell(P) + \ell(R) + d(P,R).
\end{align*}
The following estimate by Nazarov--Treil--Volberg (see e.g. \cite{NTV}) is extremely useful
\begin{displaymath}
\mathop{\sum_{P \in \dtr}}_{R \in \dtr_0} A_{PR} x_P y_R \lesssim \Big( \sum_P x_P^2 \Big)^{1/2} \Big( \sum_R y_R^2 \Big)^{1/2} 
\end{displaymath}
for every $x_P, y_R \ge 0$. For an easy reference, see pp. 159--160 in \cite{ToBook}. In particular, we have that
\begin{displaymath}
\Big(\sum_{R \in \dtr_0} \Big[ \sum_{P \in \dtr} A_{PR} x_P \Big]^2\Big)^{1/2} \lesssim \Big( \sum_P x_P^2 \Big)^{1/2}.
\end{displaymath}

The sums (1) and (2) are handled as follows. Notice that in (1) we have $\ell(P) < \ell(R) \le \ell(P_0)$ so that $\int \Delta_P f\,d\sigma = 0$.
Therefore, using the $y$-H\"older for $\tilde s_t$ we get
\begin{equation}\label{eq:est1}
|\widetilde \theta^{\sigma}_t \Delta_P f(x)| \lesssim A_{PR}\sigma(R)^{-1/2}\|\Delta_Pf\|_{L^2(\sigma)}, \qquad (x,t) \in W_R,
\end{equation}
where $W_R := R \times [\ell(R)/2, \ell(R))$.
In the case (2), the size estimate for $\tilde s_t$ yields
\begin{displaymath}
|\widetilde \theta^{\sigma}_t \Delta_P f(x)| \lesssim  \frac{\ell(R)^{\alpha}}{d(P,R)^{m+\alpha}}\sigma(P)^{1/2}\|\Delta_Pf\|_{L^2(\sigma)}, \qquad (x,t) \in W_R.
\end{displaymath}
But this yields the same bound as in \eqref{eq:est1}, since here
\begin{displaymath}
 \frac{\ell(R)^{\alpha}}{d(P,R)^{m+\alpha}}\sigma(P)^{1/2} \lesssim A_{PR}\sigma(R)^{-1/2}.
\end{displaymath}
To see this, notice that it is obvious if $d(P,R) \ge \ell(P)$. In the opposite case note that $d(P,R)^{m+\alpha} \gtrsim D(P,R)^{m+\alpha}\ell(P)^{-\alpha/2}\ell(R)^{\alpha/2}$.
This is seen by combining the facts that $d(P,R) > \ell(R)^{\gamma}\ell(P)^{1-\gamma}$, $\gamma m + \gamma \alpha = \alpha/2$ and $D(P,R) \lesssim \ell(P)$.
Thus, also in the case (2) the estimate \eqref{eq:est1} holds.
The cases (1) and (2) are therefore under control via the estimate
\begin{displaymath}
 \sum_{R \in \dtr_0} \Big[  \sum_{P \in \dtr} A_{PR} \|\Delta_Pf\|_{L^2(\sigma)} \Big]^2 
\lesssim \sum_{P \in \dtr} \|\Delta_P f\|_{L^2(\sigma)}^2 \lesssim \|f\|_{L^2(\sigma)}^2.
\end{displaymath}

The summation (3) is even easier.
Using that $P$ and $R$ are both
transit, $t \sim \ell(R) \sim \ell(P)$ and the size estimate for $\tilde s_t$ we see that
\begin{displaymath}
|\widetilde \theta^{\sigma}_t \Delta_P f(x)| \lesssim t^{-m} \sigma(P)^{1/2} \|\Delta_Pf\|_{L^2(\sigma)} \lesssim \sigma(R)^{-1/2} \|\Delta_Pf\|_{L^2(\sigma)}, \qquad (x,t) \in W_R.
\end{displaymath}
This can then easily be summed, since given $R$ there are only finitely many $P$ such that $\ell(P) \sim \ell(R)$ and $d(P,R) \lesssim \min(\ell(P), \ell(R))$.

We move on to the main term (4). For each $R \in \dtr_0$ satisfying that $R$ is $\calD$-good, $R \subset P_0$ and $\ell(R) < 2^{-r}\ell(P_0)$ we let
$P_{R,k} \in \calD$, $k \in \{r, r+1, \ldots, \log_2 [\ell(P_0)/\ell(R)]\}$, be the unique $\calD$-cube satisfying that $\ell(P_{R,k}) = 2^{k}\ell(R)$ and
$R \subset P_{R,k}$. Such a cube exists since $R$ is $\calD$-good. Moreover, since $R \not \subset H \cup T$ then also $P_{R,k} \not \subset H \cup T$ i.e. $P_{R,k} \in \dtr$.
We see that we only need to prove that
\begin{equation*}
\mathop{\mathop{\sum_{R \in \dtr_0:\, R \subset P_0}}_{R \textup{ is } \calD\textup{-good}}}_{\ell(R) < 2^{-r}\ell(P_0)} \int_R \int_{\ell(R)/2}^{\min(\ell(R), \ell(Q))}
\Big| \sum_{k=r+1}^{\log_2 [\ell(P_0)/\ell(R)]} \widetilde \theta^{\sigma}_t \Delta_{P_{R,k}} f(x) \Big|^2\, \frac{dt}{t} \,d\sigma(x) \lesssim \|f\|_{L^2(\sigma)}^2.
\end{equation*}
Recalling that all $P_{R,k}$, $r \le k \le \log_2 [\ell(P_0)/\ell(R)]$ are transit, we see using a standard calculation that
$ \sum_{k=r+1}^{\log_2 [\ell(P_0)/\ell(R)]} \widetilde \theta^{\sigma}_t \Delta_{P_{R,k}} f$ equals
\begin{align*}
 - \sum_{k=r+1}^{\log_2 [\ell(P_0)/\ell(R)]} &B_{P_{R,k-1}} \widetilde\theta^{\sigma}_t(1_{\R^n \setminus P_{R,k-1}}b) \\
&+ \sum_{k=r+1}^{\log_2 [\ell(P_0)/\ell(R)]}\widetilde \theta^{\sigma}_t (1_{P_{R,k} \setminus P_{R,k-1}}\Delta_{P_{R,k}} f) 
+ \frac{\langle f \rangle_{P_{R,r}}}{\langle b \rangle_{P_{R,r}}}\widetilde\theta^{\sigma}_t  b,
\end{align*}
where
\begin{displaymath}
B_{P_{R,k-1}} = \langle \Delta_{P_{R,k}} f / b \rangle_{P_{R,k-1}} = \left\{ \begin{array}{ll}
\frac{\langle f \rangle_{P_{R,k-1}}}{\langle b \rangle_{P_{R,k-1}}} -  \frac{\langle f \rangle_{P_{R,k}}}{\langle b \rangle_{P_{R,k}}}, & \textup{if } r+1 \le k < \log_2\frac{\ell(P_0)}{\ell(R)}, \\
\frac{\langle f \rangle_{P_{R,k-1}}}{\langle b \rangle_{P_{R,k-1}}}, & k = \log_2\frac{\ell(P_0)}{\ell(R)}. \end{array} \right.
\end{displaymath}

Let us start deciphering this by proving that the term
\begin{displaymath}
\Pi := \mathop{\mathop{\sum_{R \in \dtr_0:\, R \subset P_0}}_{R \textup{ is } \calD\textup{-good}}}_{\ell(R) < 2^{-r}\ell(P_0)} \Big|\frac{\langle f \rangle_{P_{R,r}}}{\langle b \rangle_{P_{R,r}}} \Big|^2
\int_R \int_{\ell(R)/2}^{\min(\ell(R), \ell(Q))}|\widetilde \theta^{\sigma}_tb(x)  |^2\, \frac{dt}{t} \,d\sigma(x) 
\end{displaymath}
is under control. We simply estimate
\begin{displaymath}
\Pi \lesssim \sum_{P \in \dtr} |\langle f \rangle_P|^2 a_P, \qquad a_P :=  \mathop{\mathop{\mathop{\sum_{R \in \dtr_0:\, R \subset P_0}}_{R \textup{ is } \calD\textup{-good}}}_{\ell(R) < 2^{-r}\ell(P_0)}}_{P_{R,r} = P}
\int_R \int_{\ell(R)/2}^{\min(\ell(R), \ell(Q))}|\widetilde \theta^{\sigma}_tb(x)  |^2\, \frac{dt}{t} \,d\sigma(x).
\end{displaymath}
To have $\Pi \lesssim \|f\|_{L^2(\sigma)}^2$ it is enough to verify the Carleson property of $(a_P)_{P \in \calD}$. To this end, let $S \in \calD$ be arbitrary.
We have that
\begin{align*}
\mathop{\sum_{P \in \calD}}_{P \subset S} a_P &\le \mathop{\sum_{R \in \dtr_0}}_{R \subset S} \iint_{[S \times (0, \ell(Q))] \cap W_R} |\widetilde \theta^{\sigma}_tb(x)  |^2\, \frac{dt}{t} \,d\sigma(x) \\
&\le \iint_{S \times (0, \ell(Q))} |\widetilde \theta^{\sigma}_tb(x)  |^2\, \frac{dt}{t} \,d\sigma(x) = \int_S [\widetilde V_{\sigma, Q}b(x)]^2 \,d\sigma(x) \lesssim \sigma(S),
\end{align*}
since $\widetilde V_{\sigma, Q}b(x) \lesssim 1$ for every $x \in \textup{spt}\,\sigma$ by \eqref{eq:tildeVb}.

We are only left with some completely standard calculations (but we need to be slightly careful to use transitivity).
So let us first control $|B_{P_{R,k-1}} \widetilde\theta^{\sigma}_t(1_{\R^n \setminus P_{R,k-1}}b)(x)|$ for $(x,t) \in W_R$. 
Notice that $R \subset B(x, d(R, \partial P_{R, k-1})/2)$, since
$d(R, \partial P_{R, k-1}) \ge 2^{r(1-\gamma)}\ell(R) \ge C_d\ell(R)$ by having $r$ large enough to begin with.
The point is that $B(x, d(R, \partial P_{R, k-1})/2) \not \subset H$. Moreover, we clearly have that
$B(x, d(R, \partial P_{R, k-1})/2) \subset P_{R, k-1}$. Using these facts we get
\begin{align*}
|\widetilde\theta^{\sigma}_t(1_{\R^n \setminus P_{R,k-1}}b)(x)| &\lesssim \int_{\R^n \setminus B(x, d(R, \partial P_{R, k-1})/2)} \frac{\ell(R)^{\alpha}}{|x-y|^{m+\alpha}}\,d\sigma(y) \\
&\lesssim \ell(R)^{\alpha}d(R, \partial P_{R, k-1})^{-\alpha} \lesssim \Big(\frac{\ell(R)}{\ell(P_{R, k-1})}\Big)^{\alpha/2} \sim 2^{-\alpha k/2},
\end{align*}
where we also used that $d(R, \partial P_{R, k-1}) \ge \ell(R)^{1/2}\ell(P_{R, k-1})^{1/2}$ (which follows since $R$ is $\calD$-good). Since $P_{R,k-1} \not \subset T$ we have
\begin{align*}
|B_{P_{R,k-1}}| \sigma(P_{R,k-1}) &\lesssim \Big| \int_{P_{R,k-1}} B_{P_{R,k-1}} b\,d\sigma\Big|  \\
&= \Big|\int_{P_{R,k-1}} \Delta_{P_{R,k}} f\,d\sigma\Big| \le \sigma(P_{R,k-1})^{1/2} \| \Delta_{P_{R,k}} f\|_{L^2(\sigma)}.
\end{align*}
Combining these estimates we get for $(x,t) \in W_R$ that
\begin{equation}\label{eq:same1}
|B_{P_{R,k-1}} \widetilde\theta^{\sigma}_t(1_{\R^n \setminus P_{R,k-1}}b)(x)| \lesssim 2^{-\alpha k/2} \sigma(P_{R,k-1})^{-1/2} \| \Delta_{P_{R,k}} f\|_{L^2(\sigma)}.
\end{equation}

Let us still estimate $|\widetilde \theta^{\sigma}_t (1_{P_{R,k} \setminus P_{R,k-1}}\Delta_{P_{R,k}} f)(x)|$ for $(x,t) \in W_R$. Let $S \in \textup{ch}(P_{R,k})$, $S \ne P_{R,k-1}$.
We do not know whether this cube is transitive or not, but it shall not matter. Indeed, we just estimate
\begin{align*}
|\widetilde \theta^{\sigma}_t (1_S\Delta_{P_{R,k}} f)(x)| &\lesssim \frac{\ell(R)^{\alpha}}{d(R, S)^{m+\alpha}}  \int_{P_{R,k}} |\Delta_{P_{R,k}} f(y)|\,d\sigma(y) \\
&\lesssim \Big(\frac{\ell(R)}{\ell(P_{R, k-1})}\Big)^{\alpha/2} \frac{\sigma(P_{R,k})^{1/2}}{\ell(P_{R, k-1})^m} \| \Delta_{P_{R,k}} f\|_{L^2(\sigma)} \\
&\lesssim 2^{-\alpha k/2} \sigma(P_{R,k-1})^{-1/2} \| \Delta_{P_{R,k}} f\|_{L^2(\sigma)},
\end{align*}
where we used that $\ell(S) = \ell(P_{R,k-1})$, $d(R,S)^{m+\alpha} \ge \ell(R)^{\alpha/2}\ell(S)^{\alpha/2}\ell(S)^m$ and the transitivity of $P_{R,k-1}, P_{R,k}$.
So $|\widetilde \theta^{\sigma}_t (1_{P_{R,k} \setminus P_{R,k-1}}\Delta_{P_{R,k}} f)(x)|$ satisfies the same estimate as in \eqref{eq:same1}.

We are done with the proof if we can control the summation
\begin{displaymath}
\mathop{\mathop{\sum_{R \in \dtr_0:\, R \subset P_0}}_{R \textup{ is } \calD\textup{-good}}}_{\ell(R) < 2^{-r}\ell(P_0)} \sigma(R)
 \Big[ \sum_{k=r+1}^{\log_2 [\ell(P_0)/\ell(R)]}  2^{-\alpha k/2} \sigma(P_{R,k-1})^{-1/2} \| \Delta_{P_{R,k}} f\|_{L^2(\sigma)} \Big]^2.
\end{displaymath}
Using a summation argument that appears in p. 9 in \cite{MM1} we dominate this by
\begin{displaymath}
\sum_{P \in \dtr} \|\Delta_P f\|_{L^2(\sigma)}^2 \lesssim \|f\|_{L^2(\sigma)}^2.
\end{displaymath}
\end{proof}

\end{document}